\documentclass[a4paper,10pt]{amsart}
\usepackage[centertags]{amsmath}
\usepackage[english]{babel}
\usepackage{amsfonts}
\usepackage{amssymb}
\usepackage{amsthm}
\usepackage[usenames,dvipsnames]{color}
\usepackage{newlfont}
\usepackage{graphicx}
\usepackage{epstopdf}
\usepackage{cite}
\usepackage{bm}
\usepackage{blkarray}
\usepackage{multirow}
\usepackage{multicol}
\usepackage[all,cmtip]{xy}
\usepackage{xypic}
\usepackage{multirow}
\usepackage{longtable}

\usepackage{enumerate}

\newcommand{\xdownarrow}[1]{%
    {\left\downarrow\vbox to #1{}\right.\kern-\nulldelimiterspace}}

\newcommand{\Z}{\mathbb{Z}}

\newcommand{\Q}{\mathbb{Q}}

\newcommand{\bq }{\begin{equation}}
\newcommand{\eq }{\end{equation}}
\theoremstyle{plain}
\newtheorem{thm}{Theorem}[section]
\newtheorem{lem}[thm]{Lemma}

\newtheorem{cor}[thm]{Corollary}
\newtheorem{rem}[thm]{Remark}

\theoremstyle{definition}
\newtheorem{defn}[thm]{Definition}

\theoremstyle{example}

\title{Plane model-fields of definition, fields of definition, the field of moduli of smooth plane curves}

\author[E. Badr] {Eslam Badr}
\address{$\bullet$\,\,Eslam Essam Ebrahim Farag Badr}
\address{Departament Matem\`atiques, Edif. C, Universitat Aut\`onoma de Barcelona\\
08193 Bellaterra, Catalonia, Spain} \email{eslam@mat.uab.cat}
\address{Department of Mathematics,
Faculty of Science, Cairo University, Giza-Egypt}
\email{eslam@sci.cu.edu.eg}

\author[F. Bars] {Francesc Bars}
\address{$\bullet$\,\,Francesc Bars Cortina}
\address{Departament Matem\`atiques, Edif. C, Universitat Aut\`onoma de Barcelona\\
08193 Bellaterra, Catalonia} \email{francesc@mat.uab.cat}

\thanks{E. Badr and F. Bars are supported by MTM2016-75980-P}
\begin{document}

\maketitle

\begin{abstract}
Given a smooth plane curve $\overline{C}$ of genus $g\geq 3$ over an
algebraically closed field $\overline{k}$, a field
$L\subseteq\overline{k}$ is said to be a \emph{plane model-field of
definition for $\overline{C}$} if $L$ is a field of definition for
$\overline{C}$, i.e. $\exists$ a smooth curve $C'$ defined over $L$
where $C'\times_L\overline{k}\cong \overline{C}$, and such that $C'$
is $L$-isomorphic to a non-singular plane model $F(X,Y,Z)=0$ in
$\mathbb{P}^2_{L}$.

{In this short note, we construct a smooth plane curve
$\overline{C}$ over $\overline{\mathbb{Q}}$, such that the field of
moduli of $\overline{C}$ is not a field of definition for
$\overline{C}$, and also fields of definition do not coincide with
plane model-fields of definition for $\overline{C}$.} As far as we
know, this is the first example in the literature with the above
property, since this phenomenon does not occur for hyperelliptic
curves, replacing plane model-fields of definition with the
so-called hyperelliptic model-fields of definition.
\end{abstract}

\section{Introduction}
Consider $F$ the base field for an algebraically closed field
$\overline{k}$. Let $F\subseteq L\subseteq \overline{k}$ be fields,
given a smooth projective curve $\overline{C}$ over $\overline{k}$,
then $\overline{C}$ is \emph{defined} over $L$ if and only if there
is a curve $C'$ over $L$ that is $\overline{k}$-isomorphic to
$\overline{C}$, i.e. $C'\times_L\overline{k}\cong \overline{C}$. In
such case, $L$ is called a \emph{field of definition} of
$\overline{C}$. We say that $\overline{C}$ is \emph{definable} over
$L$ if there is a curve $C'/L$ such that $\overline{C}$ and
$C'\times_L\overline{k}$ are $\overline{k}$-isomorphic.
\begin{defn}
The \emph{field of moduli} of a smooth projective curve
$\overline{C}$ defined over $\overline{k}$, denoted by
$K_{\overline{C}}$, is the intersection of all fields of definition
of $\overline{C}$.
\end{defn}
It becomes very natural to ask when the field of moduli of a smooth
projective curve $\overline{C}$ is also a field of definition. A
necessary and sufficient condition (Weil's cocycle criterion of
descent) for the field of moduli to be a field of definition was
provided by Weil \cite{We}. If $\operatorname{Aut}(\overline{C})$ is
trivial, then this condition becomes trivially true and so the field
of moduli needs to be a field of definition. It is also quite well
known that a smooth curve $\overline{C}$ of genus $g=0$ or $1$ can
be defined over its field of moduli, where $g$ is the geometric
genus of $\overline{C}$. However, if $g>1$ and
$\operatorname{Aut}(\overline{C})$ is non-trivial, then Weil's
conditions are difficult to be checked and so there is no guarantee
that the field of moduli is a field of definition for
$\overline{C}$. This was first pointed out by Earle \cite{earle} and
Shimura \cite{shimura}. More precisely, in page 177 of
\cite{shimura}, the first examples not definable over their field of
moduli are introduced, which are hyperelliptic curves over
$\mathbb{C}$ with two automorphisms. There are also examples of
non-hyperelliptic curves not definable over their field of moduli
given in \cite{hidalgo1, Hug}. B. Huggins \cite{Hug2} studied this
problem for hyperelliptic curves over a field $\overline{k}$ of
characteristic $p\neq2$, proving that a hyperelliptic curve
$\overline{C}$ of genus $g\geq2$ with hyperelliptic involution
$\iota$ can be defined over $K_{\overline{C}}$ when
$\operatorname{Aut}(\overline{C})/\langle\iota\rangle$ is not cyclic
or is cyclic of order divisible by $p$.

On the other hand, one may define fields of definition of models of
the same concrete type for a smooth projective curve $\overline{C}$.
For example, if $\overline{C}$ is hyperelliptic, a field $M$ is
called a \emph{hyperelliptic model-field of definition for
$\overline{C}$} if $M$, as a field of definition for $\overline{C}$,
satisfies that $\overline{C}$ is $M$-isomorphic to a hyperelliptic
model of the form $y^2=f(x)$, for some polynomial $f(x)$ of degree
$2g+1$ or $2g+2$. 

By the work of Mestre \cite{Me}, Huggins \cite{Hug2,Hug},
Lercier-Ritzenthaler \cite{LeRi}, Lercier-Ritzenthaler-Sijsling
\cite{LRS} and Lombardo-Lorenzo in \cite{LoLo}, one gets fair-enough
characterizations for the interrelations between the three fields;
the field of moduli, fields of definition and hyperelliptic
model-fields of definition. For instance, {if $\overline{C}$ is
hyperelliptic, then} there are always two of these fields, which are
equal. {Summing up, one obtains the next table issued from
Lercier-Ritzenthaler-Sijsling \cite{LRS}, where $k=F$ is a perfect
field of characteristic $char(F)\neq2$}:
\begin{center}
\begin{tabular}{|c|c|c|c|c|}
\hline $H=\operatorname{Aut}(\overline{C})/\langle\iota\rangle$&
Conditions& Fields of definition =&The field of moduli=\\
&&Hyperelliptic model-fields& A field of definition\\
\hline Not tamely cyclic& & Yes & Yes\\
\hline \multirow{2}{*}{Tamely cyclic with $\# H>1$}&$g$
odd,$\#H$odd&No&Yes\\
&$g$ even or $\#H$ even&Yes&No\\
\hline \multirow{2}{*}{Tamely cyclic with $\#H=1$}&$g$ odd&No&Yes\\
&$g$ even& Yes&No\\
\hline
\end{tabular}
\end{center}
By \emph{tamely cyclic}, we mean that the group is cyclic of order
not divisible by the $char(F)$.

Now, consider a smooth plane curve $\overline{C}$, i.e.
$\overline{C}$ viewed as a smooth curve over $\overline{k}$ admits a
non-singular plane model defined by an equation of the form
$F(X,Y,Z)=0$ in $\mathbb{P}^2_{\overline{k}}$, where $F(X,Y,Z)$ is a
homogenous polynomial of degree $d\geq 4$ over $\overline{k}$ with
$g=\frac{1}{2}(d-1)(d-2)\geq 3$. Similarly, we define a so-called
\emph{plane model-fields of definition for $C$}:
\begin{defn} Given a smooth plane curve $\overline{C}$ over $\overline{k}$, a subfield $M\subset\overline{k}$
 is said to be a \emph{plane model-field of definition for $C$} if and only if the following conditions holds
\begin{enumerate}[(i)]
  \item $M$ is a field of definition for $\overline{C}$.
  \item $\exists$ a smooth curve $C'$ defined over $M$, which is $\overline{k}$-isomorphic to $\overline{C}$,
  and $M$-isomorphic to a non-singular plane model $F(X,Y,Z)=0$, for some homogenous polynomial $F(X,Y,Z)\in M[X,Y,Z]$ of degree $d\geq 3$.
\end{enumerate}
\end{defn}
In this short note, we start with a smooth plane curve
$\overline{C}$ over $\overline{\Q}$ where the field of moduli is not
a field of definition by the work of B. Huggins in \cite{Hug}. Next,
we go further, following the techniques developed in \cite{BBL}, to
construct a twist of $\overline{C}$, for which there is a field of
definition for $\overline{C}$, which is not a plane model-field of
definition.
\subsection*{Acknowledgments} We would like to thank Elisa
Lorenzo and Christophe Ritzenthaler for bringing this problem to our
attention, as a consequence of our discussion with them in
BGSMath-Barcelona Graduate School in March 2017.

\section{The example}
Consider the \emph{Hessian group of order $18$}, denoted by $\operatorname{Hess}_{18}$, which is $\operatorname{PGL}_3(\overline{\mathbb{Q}})$-conjugate to the group generated by
$$S:=\left(
    \begin{array}{ccc}
      1 & 0 & 0 \\
      0 & \zeta_3 & 0 \\
      0 & 0 & \zeta_3^2 \\
    \end{array}
  \right),\,T:=\left(
                 \begin{array}{ccc}
                   0 & 1 & 0 \\
                   0 & 0 & 1 \\
                   1 & 0 & 0 \\
                 \end{array}
               \right),\,\,\text{and}\,\,R:=\left(
                           \begin{array}{ccc}
                             1 & 0 & 0 \\
                             0 & 0 & 1 \\
                             0 & 1 & 0 \\
                           \end{array}
                         \right).
$$
First, we reproduce an example, by B. Huggins in \cite[Chp. 7, \S 2]{Hug}, of a smooth $\overline{\mathbb{Q}}$-plane curve of genus $10$ not definable over its field of moduli, and with full automorphism groups $\operatorname{Hess}_{18}$.
\begin{defn}
A quaternion extension of a field $K$ is a Galois extension $K'/K$
such that $\operatorname{Gal}(K'/K)$ is isomorphic to the quaternion
group of order $8$.
\end{defn}
\begin{defn}(\cite[Lemma 7.2.3]{Hug})
A field $K$ is of level $2$ if $-1$ is not a square in $K$, but it is a sum of two squares in $K$.
\end{defn}
\begin{lem}(\cite[Lemma 7.2.3]{Hug})
Let $K$ be a field of level $2$. Then, for $u,v\in K^*\setminus(K^*)^2$ such that $uv\notin(K^*)^2$, $K(\sqrt{u},\sqrt{v})$ is embeddable into a quaternion extension of $K$ if and only if $-u$ is a norm from $K(\sqrt{-v})$ to $K$ (i.e. $-u=x^2+vy^2$ for some $x,y\in K$).
\end{lem}
For instance, the field $K:=\mathbb{Q}(\zeta_3)$ is of level $2$, since $(\zeta_3^2)^2+\zeta_3^2=-1$ and $\sqrt{-1}\notin K$. It is easily shown that $\pm2$ are not norms from $K(\sqrt{-13})$ to $K$. So neither $K(\sqrt{2},\sqrt{13})$ nor $K(\sqrt{-2},\sqrt{13})$ are embeddable into a quaternion extension of $K$.

Now fix $K$ to be the field $\mathbb{Q}(\zeta_3)$, and define the following:
\begin{eqnarray*}
\phi&:=&XYZ,\\
\psi&:=&X^3+Y^3+Z^3,\\
\chi&:=&(XY)^3+(YZ)^3+(XZ)^3.
\end{eqnarray*}
Suppose that $u,v\in\mathbb{Q}^*$, such that $L:=K(\sqrt{u},\sqrt{v})$ is a $\Z/2\Z\times\Z/2\Z$ extension of $K$ that can not be embedded into a quaternion extension of $K$. Let
\begin{eqnarray*}
c_{\phi^2}&:=&\zeta_3\sqrt{u}+\sqrt{v}+\zeta_3^2\sqrt{uv},\\
c_{\phi\psi}&:=&\zeta_3^2\sqrt{u}+\sqrt{v}+\zeta_3\sqrt{uv},\\
c_{\psi^2}&:=&\sqrt{u}+\sqrt{v}+\sqrt{uv}-\frac{1}{12}.
\end{eqnarray*}
Fix an algebraic closure $\overline{\mathbb{Q}}$ of $\mathbb{Q}$ containing $L$ as above.
\begin{thm}(B. Huggins, \cite[Lemma 7.2.5 and Proposition 7.2.6]{Hug})\label{Hugginexample3}
Following the above notations, let
$$F_{\sqrt{u},\sqrt{v}}(X,Y,Z):=c_{\phi^2}\phi^2-6c_{\phi\psi}\phi\psi-18c_{\psi^2}\psi^2+\chi.$$
Then the equation $F_{\sqrt{u},\sqrt{v}}(X,Y,Z)=0$ such that
$F_{\sqrt{u},\sqrt{v}}(X,1,1)$ is square free, defines a smooth
$\overline{\mathbb{Q}}$-plane curve $\overline{C}$ over
$\overline{\Q}$, with automorphism group $\operatorname{Hess}_{18}$.
The field of moduli $K_{\overline{C}}$ is $K=\mathbb{Q}(\zeta_3)$,
but it is not a field of definition.
\end{thm}
\begin{rem}
The condition that $F_{\sqrt{u},\sqrt{v}}(X,1,1)$ is square free is possible. For example, with $u=2$ and $v=13$, the resultant of $F_{\sqrt{2},\sqrt{13}}(X,1,1)$ and $\frac{\partial F}{\partial X}(X,1,1)$ is not zero.
\end{rem}
\begin{lem}\label{samefldmoduli}
Let $\overline{C}$ be a smooth curve defined over an algebraically
closed field $\overline{k}$, with $F=k$ and $k$ perfect. An
$\overline{k}$-isomorphism $\phi:\overline{C'}\rightarrow
\overline{C}$ does not change the field of moduli or fields of
definition, that is both $\overline{C}$ and $\overline{C'}$ have the
same fields of moduli and fields of definitions.
\end{lem}
\begin{proof}
A field $L\subseteq\overline{k}$ is a field of definition for
$\overline{C}$ if and only if there exists a smooth curve $C''$ over
$L$, such that $C''\times_L\overline{k}$ is
$\overline{k}$-isomorphic to $\overline{C}$ through some
$\psi:C''\times_L\overline{k}\rightarrow \overline{C}$. Hence
$\phi^{-1}\circ\psi:C''\times_L\overline{k}$ is a
$\overline{k}$-isomorphism, and $L$ is a field of definition for
$\overline{C'}$. The converse is true by a similar discussion.
Consequently, the field of moduli for $\overline{C}$ and
$\overline{C'}$ coincides, being the intersection of all fields of
definition.
\end{proof}
\begin{cor}\label{twisting}
Consider a smooth $\overline{\mathbb{Q}}$-plane curve $\overline{C}$
defined by an equation of the form
$$\frac{c_{\phi^2}}{p^2}(XYZ)^2-\frac{6c_{\phi\psi}}{p}(XYZ)(X^3+\frac{1}{p}Y^3+\frac{1}{p^2}Z^3)-18c_{\psi^2}(X^3+\frac{1}{p}Y^3+\frac{1}{p^2}Z^3)^2+\frac{1}{p}X^3Y^3+\frac{1}{p^3}(YZ)^3+\frac{1}{p^2}X^3Z^3=0,$$
where $p\in\mathbb{Q}$, in particular $\overline{C}$ admits
$\Q(\sqrt{u},\sqrt{v},\zeta_3)$ as a plane model-field of definition
for $\overline{C}$. Then $\operatorname{Aut}(\overline{C})$ is
isomorphic to $\operatorname{Hess}_{18}$. Moreover, the field of
moduli $K_{\overline{C}}$ is $K=\mathbb{Q}(\zeta_3)$, but it is not
a field of definition.
\end{cor}
\begin{proof}
Since $\overline{C}$ is $\mathbb{Q}(\sqrt[3]{p})$-isomorphic to
$F_{\sqrt{u},\sqrt{v}}(X,Y,Z)=0$ through a change of variables of
the shape
$\phi=\operatorname{diag}(1,1/\sqrt[3]{p},1/\sqrt[3]{p^2})$,
therefore they have conjugate automorphism groups. Moreover, fields
of definition and the field of moduli of both curves are the same by
Lemma \ref{samefldmoduli}. Consequently, the field of moduli
$K_{\overline{C}}$ is $K=\mathbb{Q}(\zeta_3)$, but it is not a field
of definition, using Theorem \ref{Hugginexample3}.
\end{proof}
\begin{thm}\label{examabst4} Consider the family $\mathcal{C}_{p}$ of smooth plane curves
over the plane model-field of definition
$L=\mathbb{Q}(\sqrt{u},\sqrt{v},\zeta_3)$ given by an equation of
the form
$$\frac{c_{\phi^2}}{p^2}(XYZ)^2-\frac{6c_{\phi\psi}}{p}(XYZ)(X^3+\frac{1}{p}Y^3+\frac{1}{p^2}Z^3)-18c_{\psi^2}(X^3+\frac{1}{p}Y^3+\frac{1}{p^2}Z^3)^2+\frac{1}{p}X^3Y^3+\frac{1}{p^3}(YZ)^3+\frac{1}{p^2}X^3Z^3=0,$$
where $p$ is a prime integer such that $p\equiv 3$ or $5$ mod $7$.
Given a smooth plane curve $C$ over $L$ in $\mathcal{C}_{p}$, then
there exists a twist $C'$ of $C$ over $L$ which does not have $L$ as a plane model-field of definition.
Moreover, the field of moduli of $C'$ is $\mathbb{Q}(\zeta_3)$, and is not a field of definition for $C'$.
\end{thm}
\begin{proof}
Consider the Galois extension 
$M'/L$ with $M'=L(cos(2\pi/7),\sqrt[3]{p})$, where all the
automorphisms of $\overline{C}:=C\times_L\overline{\mathbb{Q}}$ are
defined.
Let $\sigma$ be a generator of the cyclic 
Galois group $\text{Gal}(L(cos(2\pi/7))/L)$. We define a 1-cocycle
on $\text{Gal}(M'/L)\cong \text{Gal}(L(cos(2\pi/7))/L)\times
\text{Gal}(L(\sqrt[3]{p})/L)$ to $\text{Aut}(\overline{C})$ by
mapping $(\sigma,id)\mapsto [Y:Z:p X]$ and $(id,\tau)\mapsto id$.
This defines an element of
$\text{H}^1(\text{Gal}(M'/L),\text{Aut}(\overline{C}))$, coming from
the inflation of an element in
$\text{H}^1(\text{Gal}(L(\cos(2\pi/7))/L),\text{Aut}(\overline{C}))^{Gal(M'/L(cos(2\pi/7)))})$.

This $1$-cocycle is trivial if and only if $p$ is a norm of an
element of $L(\cos(2\pi/7)$ over $L$. However, this is not the case,
since $\Q(\cos(2\pi/7))$ and $L$ are disjoint with $[L:\Q]$ and
$[\Q(\cos(2\pi/7)):\Q]$ coprime, and moreover $p$ is not a norm of
an element of $\Q(\cos(2\pi/7))$ over $\Q$ being inert by our
assumption. Consequently, the twist $C'$ is not $L$-isomorphic to a
non-singular plane model in $\mathbb{P}^2_L$ by \cite[Theorem
4.1]{BBL}. That is, $L$ is not a plane model-field of definition for
$C'$. The last sentence in the theorem follows by Lemma
\ref{samefldmoduli} and Corollary \ref{twisting}.
\end{proof}

\begin{rem}
By our work in \cite{BBL}, we know that a non-singular plane model of $C'$ exists over at least a degree degree $3$ extension of $L$.
\end{rem}


\end{document}